\documentclass[12pt,reqno]{article}
\newcommand{\T}{\mathbf{T}}

\newcommand{\F}{\mathbb{F}}
\newcommand{\Z}{\mathbf{Z}}
\newcommand{\E}{\mathbb{E}}

\newcommand{\EM}{\mathbb{E}_{\mathbb M}}
\newcommand{\M}{\mathbb{M}}
\newcommand{\N}{\mathbf{N}}

\newcommand{\PP}{\mathbb{P}}

\usepackage{amsmath}
\usepackage{color}
\usepackage[english]{babel}
\usepackage[latin1]{inputenc}
\usepackage{bbm}
\usepackage{cite}
\usepackage{lineno}
\usepackage{amssymb}
\usepackage{fullpage}

\usepackage{graphics,amsmath,amssymb}
\usepackage{amsthm}
\usepackage{amsfonts}
\usepackage{latexsym}

\theoremstyle{plain}
\newtheorem{theorem}{Theorem}[section]
\newtheorem{lemma}[theorem]{Lemma}

\newtheorem{corollary}[theorem]{Corollary}
\newtheorem{prop}[theorem]{Proposition}

\theoremstyle{definition}

\theoremstyle{remark}

\newtheorem{observation}[theorem]{Observation}

\numberwithin{equation}{section}

\allowdisplaybreaks

\author{
Charles Burnette\\
Department of Mathematics \\
Drexel University \\
Philadelphia, PA 19104-2875 \\
cdb72@drexel.edu
\and
Eric Schmutz \\
Department of Mathematics \\
Drexel University \\
Philadelphia, PA 19104-2875 \\
Eric.Jonathan.Schmutz@drexel.edu
}
\title{Periods of Iterated   Rational Functions\vskip0cm
  over a Finite Field\\
 }
\begin{document}

\thispagestyle{empty}

\maketitle

\begin{abstract}
If $f$ is a polynomial of degree $d$ in  ${\F}_{q}[x]$, let 
 $c_{k}(f)$ be the number of cycles of length $k$ in the directed graph on ${\mathbb F}_{q}$
with edges $\lbrace (v,f(v))\rbrace_{v\in {\mathbb F}_{q}}.$
 For random polynomials, 
the numbers $c_{k}, 1\leq k\leq b,$  have asymptotic behavior resembling that for 
 the cycle lengths of   random functions $f:[q]\rightarrow [q].$ 
 %(Asymptotically independent Poisson$(1/k)$, provided  $b$ is fixed and $d=d(q)\rightarrow\infty $ as $q\rightarrow\infty$.)
However  random polynomials differ from  random functions  in  important ways.
For example,  given the set of cyclic (periodic) points, it is not necessarily true that all permutations of
those cyclic points are equally likely to occur as the restriction of $f$. 
This, and the limitations of Lagrange interpolation, 
 together complicate  research on   $\T(f)=$ the  ultimate period   of $f$ under compositional 
iteration.  We prove a lower bound for the average value of $\log \T$: if $d=d(q)\rightarrow\infty$, but $d=o(\sqrt{q})$, then
the expected value of $\log \T$ is
 \[ {\mathbb E}(\log \T):=\frac{1}{q^{d}(q-1)}\sum\limits_{f} \log \T(f)  > \frac{d}{2}(1+o(1)),\]
 where the sum is over all $q^{d}(q-1)$
polynomials of degree $d$ in ${\F}_{q}[x]$.
%No non-trivial upper bounds  are  known.
Similar results are proved for rational functions. 
%%%%%%%%%%%%%%%%%%%%%%%%%%%%%%%DELETE

\end{abstract}

\section{Introduction}

The classical theory of \lq\lq random mappings\rq\rq  concerns functions chosen randomly from among the $q^{q}$
functions whose domain and codomain  are a given $q$-element set $V.$  
 (See, for example, \cite{FO,Harris60,K}.) Basic facts about random mappings have motivated 
conjectures and theorems in \lq\lq arithmetic dynamics\rq\rq
\cite{Bach, BGHK, RodPan,Pol, Silverman,Silvermanbook}. 
This paper includes  ${\mathbb F}_{q}[x]$-analogs of known facts about  random mappings\rq\ cycle lengths and  
the  lengths of their  (ultimate) periods  under function composition.

If $f$ is any function from  ${\mathbb F}_{q}$ to  ${\mathbb F}_{q}$,   let {\sc cyclic}$(f)$  be the set of periodic points, 
i.e the set of vertices that lie on cycles in  the \lq\lq functional digraph\rq\rq\ $G_{f}$ that is formed by putting an edge from $v$ to $f(v)$
for each $v$.    Let $\mathbf{Z}(f) = |\textsc{cyclic}(f)|$ and let $\sigma_{f}$ be the restriction of $f$ to {\sc cyclic}$(f)$.
Finally,  let  ${\T}(f)$
be the least common multiple of the cycle lengths, i.e. the order of 
 $\sigma_{f}$ as an element of the group  Sym({\sc cyclic}$(f)$).
 Equivalently,  $ {\T}(f)$ is  the ultimate period of $f$, defined as the smallest 
positive integer $t$ such that,  for every    $n\geq q, f^{(n+t)}=f^{(n)}. $ %=
\vskip0cm

 A theorem of Harris\cite{Harris} tells us how large $\T$ normally is for 
 random mappings:  if $\epsilon >0$, then
for all but $o(q^{q})$ functions $f$,
\begin{equation}
\label{bounds}
      e^{(\frac{1}{8}-\epsilon)\log^{2} q} <{\T}(f) < e^{(\frac{1}{8}+\epsilon)\log^{2} q}.
\end{equation}
The average value of $\T$ is much larger \cite{EJS}. 
Igor Shparlinski and Alina Ostafe proposed the challenging problem of 
estimating the average value of  $\T$ 
in  the much harder case of  degree $d$ polynomials  and rational functions.
We require  that  $d$ be less than $q$, because otherwise all $q^{q}$ mappings can
be realized by Lagrange interpolation as  degree $d$ polynomials, and we are back to 
 random mappings.

Let $\Omega(q,d)$ be the set of all $q^{d}(q-1)$ polynomials of degree $d$
with coefficients in the finite field ${\mathbb F}_{q}$. Let $\M=\M_{q,d}$ be the uniform probability measure on 
$\Omega(q,d)$; for all $A\subseteq \Omega(q,d)$, $\M(A)=\frac{|A|}{q^{d}(q-1)}.$
Then ${\T}$, $\log \T,$  $\Z $,  and any other quantities of interest  may be regarded
 as  a random variables on  $\Omega(q,d)$, and the theorems of combinatorial probability are applicable. 
 In particular, by grouping together elements with the same period, we see that the average value of $\T$
  is just the expected value $\EM(\T)$:
\[  \frac{1}{q^{d}(q-1)}\sum\limits_{f\in \Omega(q,d)}\T(f)  =\sum\limits_{j\geq 0}j\M(\lbrace f: \T(f)=j\rbrace )= \EM(\T).\]

Here we prove, using elementary methods, that 
\[  \frac{1}{q^{d}(q-1)}\sum\limits_{f\in \Omega(q,d)}\log(\T(f))      \geq  \sum^{*}\limits_{p\leq \frac{d}{2}}
\frac{\log p}{p}\prod\limits_{j=0}^{p-1}(1-\frac{j}{q})
 -             \sum\limits^{*}_{p\leq \frac{d}{2}}    \frac{\log p}{2p^{2}}
  \prod\limits_{j=0}^{2p-1}(1-\frac{j}{q}) ,   \]
 where the $*$ indicates that the sum is restricted to prime numbers.
 As one consequence, we deduce the following crude lower bound for the average value of $\T$
 as an incomplete answer to  Shparlinski and Ostafe's question. 
 If
  $d=d(q)\rightarrow \infty$ but  $d=o(\sqrt{q})$, then 
 \[  \frac{1}{q^{d}(q-1)}\sum\limits_{f\in \Omega(q,d)}\T(f)  \geq  (1+o(1)) \frac{d}{2}.\]
% We can also prove the existence of a lower bound $\beta(q)$ such $\beta(q)\rightarrow \infty $ and  
%$\T(f)\geq \beta(q)$ for all but $o(q^{d}(q-1))$  degree-$d$ polynomials. 

In some ways, it is easier to work with random mappings than with random polynomials.
In  both Harris' proof \cite{Harris},   and in subsequent work on the average period \cite{EJS},
essential  use was made of the following
observation about classical random mappings and conditional probabilities:
given the set of cyclic vertices, all permutations of those vertices are equally likely to occur as
the restriction of $f.$  A more formal statement in terms of conditional probabilities is
\begin{observation} (Folklore)
\label{uniform} Suppose  ${\mathbb M}^{*}$ is the uniform probability measure on the $|V|^{|V|}$ functions 
$f:V\rightarrow V$. Then, 
 for all non-empty  $A\subseteq V$ 
and all permutations $\sigma$ of $A$,
\[ {\mathbb M}^{*}\left(\sigma_{f}=\sigma\, |\, \text{{\sc cyclic}}(f)=A\right) : =  \frac{{\mathbb M}^{*}\left(\sigma_{f}=\sigma\right)}{
\M^{*}(\text{\sc cyclic} (f)=A  )
}=\frac{1}{|A|!}.\]
\end{observation}
\noindent
However, a simple counting argument shows that
the analogue of Observation  \ref{uniform} cannot  be  true in general for $\Omega(d,q)$.
Consider, for example, the case where $A$ is all of $\F_{q}$.
If it were true that $\M\left(
\sigma_{f}=\sigma\, |\, \text{{\sc cyclic}}(f)=A\right) =\frac{1}{|A|!}>0$ for all $\sigma$ in $\text{Sym}(A)$,
then, for each of the $|A|!  $ permutations $\sigma$, $\Omega(q,d)$ would contain at least one element $f$
 for which $\sigma_{f}=\sigma$. This implies that  $|\Omega(q,d)|\geq  |A|! .$ But if $A$ is all of ${\F}_{q}$, and   
      if $d=o(q)$,  then $|\Omega(q,d)|< |A|! $ for all sufficiently large prime powers $q$.

Nevertheless,   $\sigma_{f}$ is always a  permutation of
{\sc cyclic}$(f).$
By an old theorem of Landau, \cite{Landau}, \cite{N1},  the maximum order 
that a permutation of an $m$-element set  can have is $e^{\sqrt{m\log m}(1+\epsilon(m))}$,
where $\epsilon(m)\rightarrow 0$.  However we do not know enough 
about the distribution of   ${\mathbf Z}$
to draw any conclusions about $ \EM(\T)$.
To date, we have only the trivial upper bound  
$ \EM(\T)\leq e^{\sqrt{q\log q}(1+o(1))}$ that follows from ${\mathbf Z}\leq q.$

\section{Small Cycle Lengths.}
Flynn and Garton\cite{FG} used Lagrange interpolation to calculate  the expected number of
cycles of length $k$.
 In this section, Flynn and Garton's  methods are combined with the  method of factorial moments
to determine the  asymptotic behaviour of  the cycle length multiplicities $c_{1},c_{2},\dots ,c_{b}$ for fixed $b$.
As we shall see, the cycle lengths are not independent, but they are asymptotically independent
as $q\rightarrow \infty$. They behave very much like the cycle lengths of random permutations 
and random mappings.

\begin{lemma}
\label{decomp}
 Let $A\subset \F_{q},$ and let $\sigma$ be a permutation of $A$.
 If $d\geq |A|,$ then there are  exactly $q^{d-|A|}(q-1)$ polynomials in $\F_{q}[x]$ of degree $d$ that extend $\sigma$
to $\F_{q}$.
\end{lemma}
\begin{proof}

Let   ${\cal E}_{d}$ consist of those  polynomials $f$  of degree $d$ that extend $\sigma$,
i.e. such that $f(x)=\sigma(x)$ for all $x\in A$.
Also let ${\cal K}_{d}$ be the set of polynomials of degree   $d$
such that $f(x)=0$ for all $x\in A$.  By the Lagrange interpolation theorem,  $f\in {\cal E}_{d}$
 if and only if 
\[ f(x)=L(x)+k(x),\]
where $L$ is the minimum degree interpolating polynomial (whose degree is less than $d$), and $k(x)\in {\cal K}_{d}.$
Since ${\mathbb F}_{q}[x]$ is a Euclidean domain, a polynomial of degree $d$
is in ${\cal K}_{d}$  if and only if it is divisible by $g(x)=\prod\limits_{a\in A}(x-a).$
Thus  $k\in {\cal K}_{d}$ if and only if there is a polynomial $h$ of degree
$d- |A| $
such that $k(x)=g(x)h(x)$. The number of ways to choose the polynomial $h$ is 
$q^{d-|A|} (q-1)$.  Therefore $|{\cal E}_{d}|=|{\cal K}_{d}|=q^{d-|A|}(q-1)$.
\end{proof}

 Before proving the next theorem, we review some basic facts from probability.
 Many readers will be familiar with the Method of  Moments. A glib summary is that,
 to prove that a sequence of
 random variables $X_{n}$ converges to $X$, it suffices
 (under  mild conditions)  to prove that, for each $r$, 
the expected  value of
 $X_{n}^{r}$ converges to the expected value of $X^{r}.$ 
  For multivariable extensions, it is necessary to consider cross moments: to prove that
  $(X_{n},Y_{n})$ converges to $(X,Y)$,  one proves that  ${\mathbb E}(X_{n}^{r_{1}}X_{n}^{r_{2}})$
  converges to
 ${\mathbb E}(X^{r_{1}}Y^{r_{2}})$  for non-negative integers $r_{1},r_{2}$.
 Section  30 of  Billingsley\cite{PB1} contains
 a more precise and detailed presentation of the Method of Moments,   together with a number-theoretic application.
Note that each of the powers $X^{r}$ can be written as a linear combination of falling the falling factorials
   \[ (X)_{s}:=\prod\limits_{j=0}^{s-1}(X-j), s=1,2,\dots ,r\]
    and vice versa.
Hence there is a corresponding \lq\lq Method of Factorial Moments\rq\rq ; convergence of the factorial moments 
implies convergence of the moments, and vice versa. 
This method is especially convenient for arithmetic and enumerative applications because
 $(X)_{r} $ occurs naturally as  the number of injective functions from an $r$-element set to an $X$-element set.   
 It is what George Andrews (\cite{Andrews}, page 32) refers to  as the
 number of $r$-permutations of an $X$ element set.
 If $X_{1}$, $X_{2}$ have a Poisson distributions with parameters $\lambda_{1}, \lambda_{2}$, i.e. if 
 $\Pr(X_{i}=m) = e^{-\lambda_{i}}\frac{\lambda_{i}^{m}}{m!}, i=1,2; m=0,1,2,\dots,$ then for any $r$ the
 $r$'th factorial moment of $X_{i}$ is $\lambda_{i}^{r}$, i.e. ${\mathbb E}((X)_{r})=\lambda_{i}^{r}.$
 With the additional hypothesis that $X_{1}$ and $X_{2}$  are independent, the falling factorials are also independent
 and the expectation of their product is the product of their expectations:
  ${\mathbb E}((X)_{r_{1}}(X)_{r_{2}})=  \lambda_{1}^{r_{1}}\lambda_{2}^{r_{2}}$. Section 6.1 of \cite{JLR} is directly relevant 
 to our application of the Method of Factorial Moments. 
 It includes (bottom of page 144) a generalization of Lemma \ref{RGlemma} below.  
%We include a proof here because it will help with understanding the proof of 
%Theorem \ref{smallcyclesthm}.  

 Before stating the lemma, we fix some notation.
 If $\sigma$ is a permutation of a subset of ${\mathbb F}_{q}$,  define an \lq\lq indicator\rq\rq (i.e. characteristic function) by $I_{\sigma}(f)=1$ if
 $\sigma$ is the restriction of $f$ to the set of elements that $\sigma$ permutes, and  $I_{\sigma}(f)=0$  otherwise.
Let  ${\cal Z}_{k}$ be the set of all  $k$-cycles that can be formed using elements of ${\mathbb F}_{q}$.
Then \[ c_{k}(f)=\sum\limits_{C\in {\cal Z}_{k}}I_{C}(f)\]
 is the number of cycles of length $k$ that $f$ has.  Also note that a product of 
indicators is itself an  indicator with a concrete interpretation:
$ I_{C_{1}}I_{C_{2}}\cdots I_{C_{r}}(f)$ is one if and only if  the $r$ cycles  $C_{1},C_{2}, \dots , C_{r}$  are all cycles of $f$, i.e
$ I_{C_{1}}I_{C_{2}}\cdots I_{C_{r}}(f)=I_{\sigma}(f),$ where $\sigma$ is the permutation having cycles $C_{1}, C_{2}, \dots ,{C_{r}}.$

\begin{lemma} 
\label{RGlemma} With the notation  $I_{C}(f)=1$ if  $C$ is a $k$-cycle for $f$,  
and zero otherwise, we have 
\[ c_{k}(c_{k}-1)\cdots (c_{k}-r+1)=\sum I_{C_{1}}I_{C_{2}}\dots I_{C_{r}},\]
 where the sum is over all sequences of  $r$ disjoint $k$-cycles.
 In other words,  there are exactly  $(c_{k}(f))_{r} $  ways to pick a sequence of  $r$ disjoint $k$ cycles of $f$.

\label{identity}
\end{lemma}

%\begin{proof}
%We use induction on $r$. The base case is a direct consequence of the definitions:
%$(c_{k})_{1}=c_{k}=\sum\limits_{C\in {\cal Z}_{k}}I_{C} .$  
%Assume, as  the inductive hypothesis, that
 %$ (c_{k})_{r-1}=\sum I_{C_{1}}I_{C_{2}}\dots I_{C_{r-1}},$
 %where the sum is over all sequences of  $r-1$ disjoint cycles.
 %We have
%\begin{align*}
%(c_{k})_{r}&=(c_{k}-r+1)(c_{k})_{r-1}\\
  %&= (\sum\limits_{C}I_{C}-r+1)(c_{k})_{r-1} \\
    %&= \sum\limits_{C,C_{1},C_{2},\dots ,C_{r-1}}I_{C}I_{C_{1}}I_{C_{2}}\cdots I_{C_{r-1}} -(r-1)(c_{k})_{r-1}.\\
    %\end{align*}
%If $C=C_{j}$ for one of the $r-1$ cycles $C_1,C_2,\dots ,C_{r-1},$ then
%$I_{C}I_{C_{j}}=I_{C_{j}}^{2}=I_{C_{j}}$. These terms exactly cancel the $(r-1)(c_{k})_{r-1}$ term. 
%If $C$ intersects one of the  $C_{j}$'s, without being equal to it, then  $I_{C}I_{C_{j}}=0$. Thus
%$ (c_{k})_{r}$ is the number of ways to choose a sequence of $r$ disjoint cycles.
%\end{proof}

Recall that ${\mathbb M}(A)=\frac{|A|}{q^{d}(q-1)}$ for all sets $A$ of degree $d$ polynomials.

\begin{theorem} 
\label{smallcyclesthm}
If $d=d(q)\rightarrow\infty$, then for any fixed $b$ the random variables $c_{k}$, $ k=1,2,\dots , b,$ 
are asymptotically independent Poisson($1/k$). In other words, for 
for any non negative integers  $m_{1},m_{2},\dots , m_{b}$,
\[ \lim\limits_{q\rightarrow\infty} {\mathbb M}(c_{k}=m_{k},1\leq k\leq b)=\prod\limits_{k=1}^{b}e^{-\frac{1}{k}}\frac{1}{m_{k}!k^{m_{k}}}.\]
\end{theorem}
\vskip.2cm\noindent
{\bf Comment:} The conclusion might well be true when $d$ is fixed. 
However the  hypothesis $d(q)\rightarrow\infty$ is needed for our proof 
because of the way Lagrange interpolation is used. 

\vskip.2cm\noindent
\begin{proof}

By [Theorem 6.10,13] (or [Theorem 21,6]),
%Theorem 6.10, page 145 of \cite{JLR} (or Theorem 21,  page 23 of \cite{Bollobas})
it suffices to show that  the joint factorial moments 
${ \EM}\left((c_{1})_{r_{1}}(c_{2})_{r_{2}}\cdots (c_{b})_{r_{b}}\right)$
converge to those of the corresponding independent Poisson distributions.
In other words, it suffices to check that, for any choice of $r_{1},r_{2},\dots ,r_{b}$,
 \begin{align*}
 \lim\limits_{q\rightarrow\infty} {\mathbb \EM}\left( \prod\limits_{k=1}^{b} (c_{k})_{r_{k}}\right)
 &=\prod\limits_{k=1}^{b}(r_{k}\text{'th factorial moment of Poisson(1/k) random var.}) \\
&= \prod\limits_{k=1}^{b} \frac{1}{k^{r_{k}}}.\end{align*}
 Let $\Lambda_{\vec{r}}$ denote a given choice of an integer partition with $r_{k}$ parts of size $k, $ and let $m=\sum_{k}kr_{k}$
 be the number that $\Lambda_{\vec{r}}$  partitions.  In the product $ \prod\limits_{k=1}^{b} (c_{k})_{r_{k}}$, we can apply
 Lemma \ref{identity} to each factor $ (c_{k})_{r_{k}}$:
 \begin{equation}
\label{b4avg}
 \prod\limits_{k=1}^{b} (c_{k})_{r_{k}}= \prod\limits_{k=1}^{b}( \sum  I_{C_{k,1}}I_{C_{k,2}}\cdots I_{C_{k,r_{k}}}).\end{equation}
If we now expand the right hand side, there are an unpleasantly  large number of terms.
On the other hand,   each term is  nothing more than a product of  indicators.
 We have $I_{C}I_{C^{'}}=0$ whenever two cycles $C$,$C'$ are not disjoint.
 Therefore, when   (\ref{b4avg}) is expanded, each non-zero term in the sum corresponds to
 a permutation $\sigma$ of of $m$ field elements having  type  $\Lambda_{\vec{r}}$ (i.e. having $r_{k}$ cycles of length $k$ for $1\leq k\leq b$).
  Because there are $r_{k}!$ possible ways to order the cycles of length $k$, 
the indicator  $ I_{\sigma}=\prod\limits_{k=1}^{b}\prod\limits_{j=1}^{r_{k}}I_{C_{k,j} } $
occurs  $\prod\limits_{k}r_{k}!$ times.
Combining the terms with same permutation we get
\begin{equation}
\label{bypartition}
 \prod\limits_{k=1}^{b}( \sum  I_{C_{k,1}}I_{C_{k,2}}\cdots I_{C_{k,r_{k}}})=(\prod\limits_{k}r_{k}!) \sum\limits_{\sigma}I_{\sigma}, \end{equation}
where the sum is over all permutations   of type $\Lambda_{\vec{r}}$  
of $m$ elements of ${\mathbb F}_{q}$.
 Averaging over all polynomials in $\Omega(q,d)$, and using the fact that  expectation ${\mathbb E}(-)$ is linear (with no assumption of independence), we get
 \begin{equation}
 \label{postavg}
 \EM( \prod\limits_{k=1}^{b} (c_{k})_{r_{k}})=(\prod\limits_{k}r_{k}!) \sum\limits_{\sigma}\EM(I_{\sigma}).
\end{equation}
Because $d(q)\rightarrow\infty$, whereas $\sigma$ is fixed permutation, we can invoke
 Lemma \ref{decomp} to get
\[ \EM(I_{\sigma})= \M(\lbrace f \text{ extends }\ \sigma\rbrace )=\frac{q^{d-m}(q-1)}{q^{d}(q-1)}=q^{-m}\]
for all sufficiently large $q$. 
Thus, in  (\ref{postavg}), all the terms in the sum are equal.
There are ${q\choose m}$ ways to choose the $m$ elements that $\sigma$ permutes.
By a well known theorem of Cauchy (see, for example, page 18 proposition 1.3.2 of Stanley \cite{Stanleyvol1}), 
an $m$-element set has exactly 
 $\frac{m!}{\prod\limits_{\ell} \ell^{r_{\ell}}r_{\ell}!}$ permutations  with the given partition $\Lambda_{\vec{r}}$
Therefore, for all sufficiently large $q$,
\[  \EM( \prod\limits_{k=1}^{b} (c_{k})_{r_{k}})=    \prod\limits_{k}r_{k}!\cdot   {q\choose m}\frac{m!}{\prod\limits_{\ell} \ell^{r_{\ell}}r_{\ell}!}q^{-m}= \frac{ (q)_{m}}{q^{m}}  \prod\limits_{k}\frac{1}{ k^{r_{k}}}.\] 
Because $m$ is fixed, 
\[ \lim\limits_{q\rightarrow \infty} \frac{ (q)_{m}}{q^{m}}=\lim\limits_{q\rightarrow \infty} \prod\limits_{j=0}^{m-1}(1-\frac{j}{q})=1.\]
Therefore   \[
\lim\limits_{q\rightarrow \infty} \EM( \prod\limits_{k=1}^{b} (c_{k})_{r_{k}})=  \prod\limits_{k=1}^{b}\frac{1}{ k^{r_{k}}}
\]
 as was to be shown.
\end{proof}
\

\section{Bounds for $ \EM(\log \T) $,  $\EM (\T)$ , and  the Typical Period.}
Recall that, for $f\in \Omega(q,d)$, $\T(f)$ is the least common multiple of
the cycle lengths, and  $c_{k}$ is the number of cycles of length $k$.  
The goal in this section is to prove lower bounds for the expected  values of 
$\T$ and $ \log \T$ as well as a lower bound for $\T$ that holds with asymptotic 
probability one. 

For any  any choice of 
 $\xi$, the period $\T$ is at least as large as the product (without multiplicity)
of prime cycle lengths in the interval  $[2,\xi]$.
Define 
$\delta_{k}(f)=1$ if $c_{k}>0$, and $\delta_{k}(f)=0$ otherwise.
Then
\begin{equation}
\label{primeproduct}
\T(f)\geq \prod\limits_{p\leq \xi }^{*}p^{\delta_{p}},\end{equation}
where the $*$ indicates that the product is restricted to primes.
Since $e^{x}>x$ for all $x\geq 0$, it is clear that $\T(f)  =e^{\log \T (f)} \geq \log \T(f).$
Combining this with (\ref{primeproduct}), and averaging over degree $d$ polynomials, we get
\begin{equation}
\label{jensen}
\EM(\T)\geq \EM(\log \T)\geq \sum\limits_{p\leq \xi}^{*}\EM(\delta_{p})\log p.
\end{equation}
We postpone the choice of $\xi$, but note that the bound can only improve
 if $\xi$ is increased.  However our ability estimate $ \EM(\delta_{p})$ will depend on
 $d$ being larger than $2\xi$, so the price for a better estimate is that the polynomials
  must have larger degree.

In order to calculate $\EM(\delta_{p}),$ we need to calculate the cardinality of 
 ${\cal Z}_{p}$, the set of all  possible $p$-cycles.  
 There are
 ${q\choose p} $  ways  to
 choose $p$ elements of ${\mathbb F}_{q}$ to form a  $p$-cycle.
  As a special case of the aforementioned theorem of Cauchy, there $(p-1)!$ ways to form a $p$-cycle from $p$ elements
  (There are $p!$ ways to write down the $p$ elements in a cycle. 
This overcounts by a factor $p$ because, calculating subscripts$\mod p$, we have  $(x_{0},x_{1},\dots ,x_{p-1})=
  (x_{i},x_{i+1},\dots , x_{i+p})$ for $0\leq i<p$ ).
 Therefore 
 \begin{equation}\label{Zpcount}
 |{\cal Z}_{p}|= (p-1)! {q\choose p}.\end{equation}

To estimate the quantity $\EM(\delta_{p})$ in equation (\ref{jensen}), note that
 $|\Omega(q,d)|\EM(\delta_{p})$ is the number of polynomials in $\Omega(q,d)$ that have at least one
 cycle of length $p$.  This number can be  estimated using  inclusion-exclusion and Bonferroni inequalities
 (See, for example, equation (7) page 66 of  \cite{Stanleyvol1}. For each possible $p$-cycle $C$, having $C$ as a 
 cycle is a property that a polynomial 
 may have.) One option is to use cardinalities of
 sets in the formulae, and divide by $|\Omega(q,d)|$ at the end to get $\EM(\delta_{p})$.  Another equivalent 
 option is  to work  directly with the probabilities as weights.)
 If $C$ is a cycle, let $A_{C}$ be the event that $f$ has $C$ as a cycle;
 $A_{C}=\lbrace f\in \Omega(q,d): I_{C}(f)=1\rbrace.$
Also define
\[ S_{r}=S_{r}(p)=\sum\limits_{\lbrace C_{1},C_{2},\dots ,C_{r}\rbrace }\M\!\left(\bigcap_{i=1}^{r}A_{C_{i}} \right)=
\sum\limits_{\lbrace C_{1},C_{2},\dots ,C_{r}\rbrace}\EM(\prod\limits_{i=1}^{r}I_{C_{i}}),
\]
where the  sums are over all  {\sl unordered} $r$ element subsets of $ {\cal Z}_{p}.$
If $p$ is prime, then by inclusion-exclusion  and  the alternating inequalities(see, for example, page 91 of \cite{Stanleyvol1}),
we have 
$  \EM(\delta_{p})=\sum\limits_{j\geq 1}(-1)^{j+1}S_{j},$ and for any $m$, 
\[   \sum\limits_{j=1}^{2m}(-1)^{j+1}S_{j}\leq   \EM(\delta_{p})\leq  \sum\limits_{j=1}^{2m-1}(-1)^{j+1}S_{j} .\]
In particular, with $m=1$ we get  a convenient lower bound:
\begin{equation}
\label{EdeltaBounds}
  S_{1}-S_{2}\leq  \EM(\delta_{p}) \leq S_{1}.
\end{equation}
So long as $d\geq p$, we can apply  Lemma \ref{decomp} to each cycle $C$ in $ {\cal Z}_{p}:$
\[ \M\!\left(A_{C}\right)=\frac{q^{d-p}(q-1)}{q^{d}(q-1)} =q^{-p}.\]
Combining this with our formula (\ref{Zpcount}) for the cardinality of ${\cal Z}_{p}$, we get
 an exact formula for  $S_{1}:$
\begin{equation}
\label{S1id}
 S_{1}=\sum\limits_{C\in {\cal Z}_{p}} q^{-p}= {q\choose p}\frac{(p-1)!}{q^{p}}= \frac{1}{p} \frac{(q)_{p}}{q^{p}}= \frac{1}{p}\prod\limits_{j=0}^{p-1}(1-\frac{j}{q}).\end{equation}
Similarly, if $d\geq 2p$, we can apply Lemma \ref{decomp} to any permutation that consists of
two disjoint $p$-cycles. (As before, intersecting cycles contribute zero).   Therefore $S_{2}=\sum\limits_{\lbrace C_{1},C_{2}\rbrace}q^{-2p}$, where the sum is restricted to disjoint pairs of cycles.
For the  number of ways to choose a set of two disjoint $p$-cycles, we
again have ${q\choose 2p}$ choices of field elements to permute.
Again, by Cauchy's theorem, there are $\frac{(2p)!}{2p^{2}}$ permutations of
of those elements that have two $p-$cycles.
Thus
 \begin{equation}\label{S2id}
  S_{2}= \frac{(2p)!}{2p^{2}}{q\choose 2p}q^{-2p}=\frac{1}{2p^{2}}{\prod\limits_{j=0}^{2p-1}}(1-\frac{j}{q}).
  \end{equation}  
  Putting (\ref{S1id}) and (\ref{S2id}) into 
(\ref{EdeltaBounds}) and then  (\ref{jensen}), we get
\begin{theorem} 
\label{ETtheorem}
If $d\geq 2\xi ,$ then
\[  \EM(\log \T) \geq  \sum^{*}\limits_{p\leq \xi}
\frac{\log p}{p}{\prod\limits_{j=0}^{p-1}}(1-\frac{j}{q})
 -             \sum\limits^{*}_{p\leq \xi}    \frac{\log p}{2p^{2}}
  \prod\limits_{j=0}^{2p-1}(1-\frac{j}{q}) ,   \]
 where the $*$ indicates that the sum is restricted to prime numbers.
  \end{theorem}
 
 Clearly the second sum  in Theorem \ref{ETtheorem}  is bounded by an absolute constant, regardless of
   how large $d$ and $\xi$ are, since the unrestricted series $\sum\limits_{n=1}^{\infty}\frac{\log n}{n^{2}}$
   is convergent.  For the main term, one can make various asymptotic approximations,
   depending on the choice of $d$ and $\xi$.  For example, if $d=d(q)\rightarrow \infty$ but $d =o(\sqrt{q})$,  then we can
 choose  $\xi=\frac{d}{2}$ and use  the approximation $\log(1-u)=O(u)$ to simplify the product:
 \begin{equation}
 \label{logapprox}
    \exp(\sum\limits_{j=1}^{p-1}\log(1-\frac{j}{q}))
   =\exp( \sum\limits_{j=1}^{p-1}O(\frac{j}{q})) =e^{o(1)}=1+o(1).\end{equation}
It is well known fact (\cite{Hua}, page 89), reportedly due to Mertens,  that $\sum\limits^{*}_{p\leq \xi}\frac{\log p}{p}=\xi+O(1)$
as $\xi\rightarrow\infty.$
 Therefore we have
 \begin{corollary} 
 \label{logcor}
 If  $d=d(q)\rightarrow \infty$, but $d =o(\sqrt{q})$,then 
 \[  \EM(\log \T)\geq  (1+o(1))\frac{d}{2}.\]   
\end{corollary}
\vskip.5cm
\noindent

Corollary  \ref{logcor} does not necessarily mean that most polynomials have
 period greater $e^{(\frac{1}{2}-\epsilon)d}$.  Without more information, one   cannot even rule out the 
existence of a constant upper bound $\kappa$ such that 
$\T \leq \kappa$ for all but $o(q^{d}(q-1))$ polynomials in $\Omega(q,d)$. 
However the next proposition shows that, in fact, most polynomials
have order larger than  $\sqrt{\frac{d}{2}}.$
\vskip.2cm
\begin{prop} If $d=d(q)\rightarrow \infty$ but $d =o(\sqrt{q})$, then all but 
$o(q^{d}(q-1)$)  polynomials in  $\Omega(q,d)$ have at least one cycle whose length 
 is in the interval $J=[\beta ,\beta^{2}]$, where $ \beta(q)=\sqrt{\frac{d}{2}}.$
 \end{prop}
\vskip.2cm
\begin{proof}
For $f\in \Omega(q,d)$, let 
$\N(f)=\sum\limits_{k\in J}c_{k}$ be the number 
of cycles of $f$ that have  length in $J$. The goal is  to show 
$\M(\N=0)=o(1)$.  If $\N=0$, then obviously $|\N-\mu|\geq \mu$ for any real number $\mu$.  Therefore
\begin{equation}
\label{setupChebyshev}
\M(\N=0)\leq   \M( |\N-\mu|\geq \mu).
\end{equation}
A standard approach is to set $\mu:=\EM(\N ),$ and $\sigma^{2}:=\EM( (\N-\mu)^{2}),$
and use Chebyshev's inequality(see, for example, page 75 of [5])
to show that the right side of (\ref{setupChebyshev}) approaches zero.
By  Chebyshev's inequality,
\begin{equation}
\label{chebsyshev}
  \M(|\N-\mu|\geq \mu) \leq \frac{\sigma^{2}}{\mu^{2}}.\end{equation}
It therefore suffices to show that $ \sigma^{2}=o(\mu^{2}).$   
 By elementary algebra,
\begin{equation}
\label{algebra}
 %(\N-\mu)^{2}=\N(\N-1)+\N-\mu^{2}.
 (\N-\mu)^{2}=\N(\N-1)+(1-2\mu)\N + \mu^{2}.
 \end{equation}
 Now average over polynomials in $\Omega(q,d)$;
take the expected value of both sides of  (\ref{algebra}), to get
\begin{equation}
\label{variance}
\sigma^{2}=\EM((\N)_{2})+(1-2\mu)\mu + \mu^{2} =\EM((\N)_{2})+\mu -\mu^{2}.
\end{equation}
To calculate $\mu$, use Lemma \ref{decomp}   just like before in (\ref{S1id}). (See also  \cite{FG}).
We chose $\beta$ so that  $d=2\beta^{2}\geq k$, therefore Lemma \ref{decomp} applies.
If ${\cal Z}_{k}$ denotes the set of all possible $k$-cycles, then 
$\N=\sum\limits_{k\in J}\sum\limits_{C\in {\cal Z}_{k}}I_{C}, $ and 
\[ \mu=   \sum\limits_{k\in J}\sum\limits_{C\in {\cal Z}_{k}}\EM(I_{C}) =  \sum\limits_{k\in J}|{\cal Z}_{k}| q^{-k}   = \sum\limits_{k\in J}\frac{(q)_{k}}{q^{k}}\frac{1}{k}.  \]
Since $k\leq d=o(\sqrt{q})$, we get 
$\frac{(q)_{k}}{q^{k}}= \prod\limits_{j=0}^{k-1}(1-\frac{j}{q})=1+o(1)$, just as in  (\ref{logapprox}). 
Thus
\begin{equation} \mu =(1+o(1))\sum\limits_{k=\beta}^{\beta^{2}}\frac{1}{k}=(1+o(1))\log \beta.\end{equation}

As in Lemma \ref{RGlemma}, we have
$(\N)_{2}=\sum\limits_{C_{1},C_{2}}I_{C_{1}}I_{C_{2}}$ where the sum is over distinct pairs of
disjoint cycles $C_{1},C_{2}$ whose lengths lie in $J$.   If $C_{1}$ and $C_{2}$ are disjoint cycles of
lengths $k_{1},k_{2}$ respectively, then by Lemma \ref{decomp},  $\EM(I_{C_{1}}I_{C_{2}})=
q^{-k_{1}-k_{2}}$. (We chose $\beta$ so that  $d=2\beta^{2}\geq k_{1}+k_{2}$.) The number of ways to choose an ordered pair of 
disjoint cycles of length $k_{1}$ and $k_{2}$
is $\frac{q!}{k_{1}!k_{2}!(q-k_{1}-k_{2})!}(k_{1}-1)!(k_{2}-1)!= \frac{(q)_{k_{1}+k_{2}}}{k_{1}k_{2}}$ 
just as in the proof of Theorem \ref{smallcyclesthm}.
Hence 
\[ \EM((\N)_{2})= \sum\limits_{k_{1},k_{2}\in J} \frac{(q)_{k_{1}+k_{2}}}{ q^{k_{1}+k_{2}}}  \frac{1}{k_{1}k_{2} }.  \]
Because $\beta^{2}=o(\sqrt{q})$, we have $ \frac{(q)_{k_{1}+k_{2}}}{ q^{k_{1}+k_{2}}} =1+o(1)$ for
all $k_{1},k_{2}\in J$.  Therefore
\[ \EM((\N)_{2})=(1+o(1))\sum\limits_{k_{1},k_{2}\in J}\frac{1}{k_{1}k_{2}} =(1+o(1))(\log\beta)^{2} = \mu^{2}(1+o(1)).\]
Putting this back into (\ref{variance}), we get
$ \sigma^{2}=o(\mu^{2}).$
By (\ref{chebsyshev}), this completes the proof.
\end{proof}

\section{Rational Functions}
If $f$ and $g$ are polynomials in ${\mathbb F}_{q}[x]$, let $\rho(g)=$ the degree of $g$ and let  $mgcd(f,g)$
be the greatest monic common divisor of $f$ and $g$.
Let \[U(q,d)=\lbrace (f,g): \rho(f)=\rho(g)=d \text{ and } mgcd(f,g)=1 \text{ and } g \text{ is monic}\rbrace.\]
 It is known \cite{BenBen}, \cite{Wilf} that 
 \begin{equation}
 \label{wilf}|U(q,d)|=q^{2d+1}(1-\frac{1}{q})^{2}.
 \end{equation}

For each pair $(f,g)\in U(q,d)$, the 
 rational function $R(x)=\frac{f(x)}{g(x)}$ can be regarded as a function from $\F_{q} \cup\lbrace\infty\rbrace$ 
 to $\F_{q}\cup\lbrace\infty\rbrace$  with the convention that $R(x)=\infty$ whenever $g(x)=0$ and
$R(\infty)=$ the leading coefficient of $f$. Note that, because of the conditions on $f$ and $g$,
the polynomials $f$ and $g$ are uniquely determined by the rational function $R$. 
Let $\PP_{d}$ be the uniform probability measure on $U(q,d)$, and let ${\mathbb E}_{d}$ denote the corresponding 
expectation.  Since $F_{q}\cup\lbrace\infty\rbrace$ is finite, the   ultimate period  is  well-defined.
Rather than introduce  new notation, we
reuse the same notation $\T$ for the ultimate period; in this section $\T$ defined on $U(q,d)$ instead of $\Omega(q,d)$.

Our goal in this section is to prove a lower bound for
$\E_{d}(\log \T)$. To avoid dealing with the exceptional point $\infty$, define
$\hat{c}_{k}$ to be the number of cycles of length $k$ that do not include $\infty$.
Also let $\hat{\delta}_{k}=1$ if $\hat{c}_{k}>0$,   $\hat{\delta}_{k}=0$ otherwise.
With this notation,   $\T$ may be strictly larger than the least common multiple of the $k$'s
for which $\hat{c}_{k}>0$. (It is possible that the only cycle of length $k$ happens to contain $\infty$.)
  That poses no difficulties because we are proving lower bounds. As before, we have

\begin{equation}
\label{primeproduct2}
\T(f)\geq \prod\limits_{p\leq \zeta }^{*}p^{\hat{\delta}_{p}},\end{equation}
where the $*$ indicates that the product is restricted to primes, and $\zeta$ is a parameter to be specified later.

Averaging over rational functions in $U(q,d)$, we get
\begin{equation}
\label{Rjensen}
\E_{d}(\T)\geq \E_{d}(\log \T)\geq \sum\limits_{p\leq \zeta }^{*}\E_{d}(\hat{\delta}_{p})\log p,
\end{equation}
where the choice of the parameter $\zeta$ is postponed.
Next we prove a rational function analogue of lemma \ref{decomp}.
\begin{lemma}
\label{decomp2}
 Let $A\subset \F_{q} $, and let $\sigma$ be a permutation of $A$.
 If $d\geq 2|A|,$ then
\[ \PP_{d}(\lbrace (f,g)\in U(q,d):  \frac{f}{g} \text{ extends } \sigma\rbrace )= q^{-|A|}\!\left(1+O\!\left(\frac{|A|}{q}\right)\right).\]
\end{lemma}

\begin{proof}
From the definitions we have
\[ \PP_{d}(\lbrace (f,g) \in U(q,d):  \frac{f}{g}  \text{ extends } \sigma\rbrace )=\]
\begin{equation}
\label{unrelaxed}
\frac{1}{|U(q,d)|  }\sum\limits_{g}|\lbrace f: \frac{f}{g}\in U(q,d)  \text{ and } \frac{f(x)}{g(x)}=\sigma(x) \text{ for all } x\in A \rbrace |, 
\end{equation}
where the sum is over monic degree $d$   polynomials $g$ that have no roots in $A$.
As an auxiliary device,  consider a superset of $U(q,d)$  in which the coprimality  restriction is removed:
\[ U^{*}(q,d)= \lbrace (f,g): \rho(f)=\rho(g)=d \text{ and } g \text{ is monic}\rbrace.
\] 
If $(f,g)\in U^{*}(q,d)$, and $mgcd(f,g)=h$ has degree $k$, then  $f^{*}:=\frac{f}{h}$ and $g^{*}:=\frac{g}{h}$
are coprime polynomials of degree $d-k$.   Define 
\[ \Phi_{1}=\frac{1}{|U(q,d)|  }\sum\limits_{g}|\lbrace f: \frac{f}{g}\in U^{*}(d,q)  \text{ and } {\frac{f}{g}}=\sigma(x) \text{ for all } x\in A \rbrace |, \]
where, as in (\ref{unrelaxed}), the sum is over monic degree $d$   polynomials $g$ that have no roots in $A$.
Comparing $\Phi_{1}$ with (\ref{unrelaxed}), we see that $\Phi_{1}$ is an overestimate because $f$ is chosen from a larger set.
We can therefore write
\begin{equation}
\label{sieve}
 \PP_{d}(\lbrace R\in U(q,d): R \text{ extends } \sigma\rbrace )= \Phi_{1}-\Phi_{2}-\Phi_{3},
\end{equation}
where
$\Phi_{2}$ is the excess contribution from  those pairs in ${U}^{*}(q,d)-U(q,d)$  for which 
the degree of the mgcd is is  larger than one, but not larger than $|A|$,
and $\Phi_{3}$ is the remaining  
excess contribution from pairs with for which the $mgcd$ has degree larger than $|A|$.   
To fit our equations on one line, abbreviate $U=|U(q,d)|$.
Then we can be more explicit:
\[ \Phi_{2}=    \frac{1}{U } 
\sum\limits_{\lbrace h: 2\leq \rho(h)\leq |A| \rbrace}|\lbrace (f,g): mgcd(f,g)=h,g \text{ monic, w.o. roots in } A,
{\frac{f}{g}}\ \text{extends } \sigma\rbrace|,
\]
and 
\[ \Phi_{3}=    \frac{1}{|U|  } 
\sum\limits_{\lbrace h: \rho(h)> |A| \rbrace}|\lbrace (f,g): mgcd(f,g)=h,g \text{ monic, w.o. roots in } A,
{\frac{f}{g}}\ \text{extends } \sigma\rbrace|.
\]
First we estimate the main term $\Phi_{1}$. In $\Phi_{1}$, the number of terms in the sum  is 
 the number of degree $d$ monic polynomials with no root in $A$.
We use inclusion-exclusion, with property $i$ being that $g$ has the $i$'th  element of $A$ as a root.
For any particular set $\alpha_{1},\alpha_{2}, \dots \alpha_{j}$ of roots, the number of 
monic degree $d$ polynomials having those roots is the number    that can written as
$(x-\alpha_{1})(x-\alpha_{2}) \cdots (x-\alpha_{j}) \tilde{g}(x)$ for some 
monic degree $d-j$ poly $\tilde{g}$, namely $q^{d-j}.$ 
By inclusion-exclusion, the number of terms in the sum $\Phi_{1}, $ i.e.  the number of degree $d$ monic polynomials 
with no root in $A$, is
\[ \sum\limits_{j=0}^{|A|}{|A| \choose j}(-1)^{j}q^{d-j} =
q^{d}(1-\frac{1}{q})^{|A|}= q^{d}(1+O(\frac{|A|}{q})).\] 
Because $d> |A|$, we can use Lemma \ref{decomp} 
 again to calculate the terms in the sum $\Phi_{1}$  for a given $g$:
\[ |\lbrace f: \rho(f)=d \text{ and } f(x)=\sigma(x)g(x) \text{ for all } x\in A  \rbrace | = q^{d-|A|}(q-1).\]
  Combining our estimates for the size of terms, the number of terms, and  the cardinality of $U(q,d)$,   we get
$\Phi_{1}= \left(  \frac{1}{  q^{2d+1}(1-\frac{1}{q})^{2}}        \right)   \left(  q^{d}(1-\frac{1}{q})^{|A|}  \right)\left(   q^{d-|A|}(q-1)   \right).$
Thus
\begin{equation}\label{T1}
\Phi_{1}
=q^{-|A|}(1+O(\frac{|A|}{q})).
\end{equation}

Next we show that  $\Phi_{2}$ is negligible compared to $\Phi_{1}$ . If the degree 
of $h= mgcd(f,g)$ is $k$, then $f^{*}:=f/h$ and $g^{*}:=g/h$ have degree $d-k$.
 Given $k$, there are $q^{k}$ choices of a monic polynomial   $h=mgcd(f,g)$,
 and at most $q^{d-k}$ choices of a monic polynomial  $g^{*}$ of degree $d-k$.
We have $d\geq 2|A|$ as a hypothesis, and we have $k\leq |A|$  built into the definition of $\Phi_{2}$.
Therefore  $d-k\geq |A|$ and  Lemma \ref{decomp} is applicable: there are, for a given $g^{*}$, at most $q^{d-k-|A|}$ choices for
$f^{*}$ such that   $f^{*}(x)=\sigma(x)g^{*}(x)$ for all $x\in A$. 
Therefore
\begin{align*}
 \Phi_{2} &\leq   \frac{1}{|U(q,d)|  } \sum\limits_{k=2}^{|A|}q^{k} q^{d-k}  q^{d-k-|A|}  \\
 &=
\frac{q^{2d-|A|}}{ q^{2d+1}(1-\frac{1}{q})^{2}} \sum\limits_{k=2}^{|A|}  q^{-k}  \\
&= q^{-|A|}O(\frac{1}{q^{3}})=O(\frac{\Phi_{1}}{q^{3}}).
\end{align*}
Thus
\begin{equation}
\label{T2}
\Phi_{2}=O(\frac{\Phi_{1}}{q^{3}}).
\end{equation}

 % \frac{1}{ q(1-\frac{1}{q})^{2}} \sum\limits_{k=2}^{|A|}  q^{-k-|A|} =O\!\left(\frac{\Phi_{1}}{q^{3}}\right).  \end{equation}

We can be very crude about estimating $\Phi_{3}$.  We get a satisfactory bound by  ignoring completely the requirement that $\frac{f^{*}}{g^{*}}$
extends $\sigma$, and $g$ has no roots in $A$. We retain only the requirements that $g$ and $h$ are monic, and that 
$\rho(f^{*})=\rho(g^{*})=d-\rho(h).$ Thus
%\[  \Phi_{3}=    \frac{1}{|U(q,d)|  } 
%\sum\limits_{\lbrace h: \rho(h)> |A| \rbrace}|\lbrace (f,g): mgcd(f,g)=h,g \text{ monic, w.o. roots in} A,
  %\overline{\frac{f}{g}}\ \text{extends } \sigma\rbrace|
%\]

\[  \Phi_{3} \leq   \frac{1}{|U(q,d)|  } 
\sum\limits_{\lbrace \text{monic } h: \rho(h)> |A| \rbrace}|\lbrace (f^{*},g^{*}): \rho(f^{*})=\rho(g^{*})=d-\rho(h) , g^{*} \text{\ monic}
\rbrace|.
\]
Given $k>|A|,$ there are $q^{k}$ choices of a monic polynomial $h$ of degree $k$, $q^{d-k}$ choices of a monic polynomial 
$q^{*}$ of degree $d-k$, and $q^{d-k}(q-1)=q^{d-k+1}(1-\frac{1}{q})$ choices of a polynomial $f^{*}$ of degree $d-k$.
Therefore 
 \begin{align*} \Phi_{3} &\leq \frac{1}{|U(q,d)|  } \sum\limits_{k=|A|+1}^{d}q^{k}q^{d-k} q^{d-k+1}(1-\frac{1}{q})\\
     &= \frac{1}{q^{2d+1}(1-\frac{1}{q})^{2}}\sum\limits_{k=|A|+1}^{d}q^{2d+1-k}(1-\frac{1}{q})
     &=  \frac{1}{(1-\frac{1}{q})}\sum\limits_{k=|A|+1}^{d}q^{-k}=  O\!\left(\frac{1}{q^{|A|+1}}\right).\end{align*}
Thus
\begin{equation}
\label{T3}
\Phi_{3}=O\left(\frac{\Phi_{1}}{q}\right).\end{equation}
Putting  (\ref{T1}), 
(\ref{T2}), and (\ref{T3}) into (\ref{sieve}), we get lemma (\ref{decomp2}).
\end{proof}

With Lemma \ref{decomp2} in hand, the proof of a lower bound for $\E_{d}(\log \T)$ is essentially the same as in 
the preceding section.  We need  a new notation for the rational function analogues of  $S_{1}$ and $S_{2}$.
Let
\[ \hat{S}_{r}=\hat{S}_{r}(p)=\sum\limits_{\lbrace C_{1},C_{2},\dots ,C_{r}\rbrace }{\mathbb P}_{d}\!\left(\bigcap_{i=1}^{r}A_{C_{i}} \right)=
\sum\limits_{\lbrace C_{1},C_{2},\dots ,C_{r}\rbrace}{\mathbb E}_{d}(\prod\limits_{i=1}^{r}I_{C_{i}}),
\]
(Cycles that contain $\infty$ are not included in this sum. As, before ${\cal Z}_{p}$ is the set of $p$-cycles that can be formed from the elements of ${\mathbb F}_{q}$.)
Taking $\sigma$  in Lemma \ref{decomp2} to be a   $p$-cycle, we can estimate the sum  $\hat{S}_{1}.$
If $d\geq 2p$,
\begin{align*} \hat{S}_{1}&=\sum\limits_{C\in {\cal Z}_{p}}\PP_{d}(C\ \text{is a cycle of } R) \\
&= {q\choose p}(p-1)!q^{-p}\left(1+O\!\left(\frac{p}{q}\right)\right).
\end{align*}
Thus
 \begin{equation}
 \label{S1Re}
  \hat{S}_{1} = \frac{1}{p}\prod\limits_{j=0}^{p-1}(1-\frac{j}{q})\left(1+O\!\left(\frac{p}{q}\right)\right)=\frac{1}{p}\prod\limits_{j=0}^{p-1}(1-\frac{j}{q})+O(\frac{1}{q}).
 %=\frac{1}{p}+O\!\left(\frac{1}{p^{2}}\right).
 \end{equation}
% where  $\hat{E}_{1}=O(\frac{1}{1q}\prod\limits_{j=0}^{p-1}(1-\frac{j}{q} ))= O(\frac{1}{q}) $
Similarly, if $d\geq 4p$, we can apply Lemma \ref{decomp2} to any permutation that consists of
two $p$-cycles.  Therefore 
\begin{align*} 
\hat{S}_{2}&=\sum\limits_{\lbrace C_{1},C_{2}\rbrace}{\mathbb P}_{d}(C_{1},C_{2}\text{ are cycles})\\
               &= \frac{q(q-1)\cdots (q-2p+1)}{2p^{2}}  q^{-2p}\left(1+O\!\left(\frac{p}{q}\right)\right)
\end{align*}
After simplification, this is 
\begin{equation}
\label{S2Re} \hat{S}_{2}=\frac{1}{2p^{2}}{\prod\limits_{j=0}^{2p-1}}(1-\frac{j}{q})+O(\frac{1}{pq})
%\hat{E}_{2},
%\left(1+O\!\left(\frac{p}{q}\right)\right).
\end{equation}
%where $\hat{E}_{2}=O( \frac{1}{qp}))=O(\frac{1}{q}).$ 
As in (\ref{EdeltaBounds}), we have 
\begin{equation}
\label{REdeltaBounds}
  \hat{S}_{1}-\hat{S}_{2}\leq {\mathbb E}_{d}(\hat{\delta}_{p}) \leq \hat{S}_{1}.
\end{equation}
Putting (\ref{S1Re}) and (\ref{S2Re}) into 
(\ref{REdeltaBounds}) and then  (\ref{jensen}), we get  
\[  {\mathbb E}_{d}(\log \T) \geq  \sum^{*}\limits_{p\leq \zeta}
\frac{\log p}{p}{\prod\limits_{j=0}^{p-1}}(1-\frac{j}{q}) 
- \sum\limits^{*}_{p\leq \zeta}    \frac{\log p}{2p^{2}}{\prod\limits_{j=0}^{2p-1}}(1-\frac{j}{q}) +
O(\sum\limits_{p\leq \zeta}\frac{\log p}{q}).\]
If we take $\zeta=\frac{d}{4}$, and require $d=o(\sqrt{q})$, then 
$\sum\limits_{p\leq \zeta}\frac{\log p}{q} =o(1)$ and, as in (\ref{logapprox}), the products 
are $1+o(1)$. We therefore have 
a rational function analogue of  Corollary \ref{logcor}:
%=O(\frac{\xi^{2}}{q\log \xi})$
\begin{theorem}  If $d=d(q)\rightarrow\infty$ is such a way that $d=o(\sqrt{q})$, then
\label{ETtheorem2} 
\[  {\mathbb E}_{d}(\log \T) \geq \frac{d}{4}(1+o(1)).\]
  \end{theorem}
  %%%%%%%%%%%%%%%%%%%%%%%%%%
\vskip0cm

\vskip.5cm\noindent
{\bf Acknowledgement:} We thank Igor Shparlinski  and BIRS, for providing reference 
\cite{FG} and stimulating our interest in this subject.   We also
thank Derek Garton for helpful comments. A referee's suggestions helped us  
improve both the proofs and the exposition.

\bibliographystyle{plain}

\begin{thebibliography}{99}

\bibitem{Andrews}  {Andrews, George E.}, {Number theory},
  (Corrected reprint of the 1971 original)  {Dover Publications, Inc., New York},
  {1994},   ISBN {0-486-68252-8}
   
\bibitem{Bach}
 {Bach, Eric},
   {Toward a theory of {P}ollard's rho method},
 {Inform. and Comput.},
  {\sl Information and Computation},
 {\bf 90}, (1991),  {no.2}, {139--155}.   


\bibitem{BGHK}
 {Benedetto, Robert L. and Ghioca, Dragos and Hutz, Benjamin and
              Kurlberg, P{\"a}r and Scanlon, Thomas and Tucker, Thomas J.},
   {Periods of rational maps modulo primes},
 {\sl Math. Ann.}, {\bf 355},  {(2013)}, {no.2},{637--660}.

\bibitem{BenBen}
 {Benjamin, Arthur T. and Bennett, Curtis D.},
{The probability of relatively prime polynomials},
 {\sl Math. Mag.}, {\bf 80}, {2007}, {no.3},
 {196--202}.
 
 \bibitem{PB1}{Billingsley, Patrick},
  {Probability and measure},second edition,
   {Wiley Series in Probability and Mathematical Statistics:
              Probability and Mathematical Statistics},
  {John Wiley \& Sons, Inc., New York},
 {1986}, ISBN {0-471-80478-9}.
 
 \bibitem{Bollobas} {Bollob{\'a}s, B{\'e}la},
 {Random graphs},
 {Academic Press, Inc. [Harcourt Brace Jovanovich, Publishers],
              London},
    {1985}, ISBN = {0-12-111755-3; 0-12-111756-1}.
     
\bibitem{Wilf}
 {Corteel, Sylvie and Savage, Carla D. and Wilf, Herbert S. and
              Zeilberger, Doron},
 {A pentagonal number sieve},
 {\sl J. Combin. Theory Ser. A},
  {\bf 82}
   ({1998}), {no. 2},
 {186--192}.



%\bibitem{ET}
 %{Erd{\H{o}}s, P. and Tur{\'a}n, P.},
%{On some problems of a statistical group-theory. {II}},
%{\sl Acta Mathematica Academiae Scientiarum Hungaricae},
   % Vvol. {\bf 18},
   %{1967},
    %{151--163}.

\bibitem{FO}  {Flajolet, Philippe and Odlyzko, Andrew M.},
 {Random mapping statistics},
in  {Advances in cryptology---{EUROCRYPT} '89 ({H}outhalen, 1989)},
 {Lecture Notes in Comput. Sci.},
 {\bf 434},
 {329--354},
 {Springer, Berlin},
  {1990}.
 
\bibitem{FG} {Flynn, Ryan and Garton, Derek},
{Graph components and dynamics over finite fields},
 {\sl Int. J. Number Theory},
 {\bf 10}, {2014}, {no. 3}, {779--792}.


\bibitem{Harris60}
 {Harris, Bernard},
   {Probability distributions related to random mappings},
 {\sl Ann. Math. Statist.},
   {1960}, {1045--1062}.

\bibitem{Harris}
{Harris, Bernard},
 {The asymptotic distribution of the order of elements in
              symmetric semigroups},
 {\sl J. Combinatorial Theory Ser. A},
 {\bf 15}, {(1973)},
  {66--74}.

\bibitem{Hua}
 {Hua, Loo Keng},
 {Introduction to number theory},
       {Translated from the Chinese by Peter Shiu},
  {Springer-Verlag, Berlin-New York},
 {1982}.
    
\bibitem{JLR}  {Janson, Svante and {\L}uczak, Tomasz and Rucinski, Andrzej},
     {Random graphs},
   {Wiley-Interscience Series in Discrete Mathematics and
              Optimization},
 {Wiley-Interscience, New York},
      {2000}, ISBN  {0-471-17541-2}.
      
\bibitem{K} {Kolchin, Valentin F.},  {Random mappings},
 {Translation Series in Mathematics and Engineering}
    {Optimization Software, Inc., Publications Division, New York},
   {1986},ISBN {0-911575-16-2},
   
\bibitem{Landau} {Landau, Edmund},
  {Handbuch der {L}ehre von der {V}erteilung der {P}rimzahlen. 2
    {B}\"ande},{2d ed}, {Chelsea Publishing Co., New York}, {1953}.
  
\bibitem{RodPan} Martins, Rodrigo S. V. and  Panario, Daniel, 
On the Heuristic of Approximating Polynomials over Finite Fields by Random Mappings,
arXiv:1505.02983.
\bibitem{N1}  {Massias, J.-P. and Nicolas, J.-L. and Robin, G.},
 {\'{E}valuation asymptotique de l'ordre maximum d'un
              \'el\'ement du groupe sym\'etrique},
 {\sl Acta Arith.}, {\bf 50},   {1988}, no. {3}, {221--242}.
      

	
\bibitem{Pol}
 {Pollard, J. M.}, {A {M}onte {C}arlo method for factorization},
 {\sl Nordisk Tidskr. Informationsbehandling (BIT)},
  {\bf 15},  (1975), {no.3}, {331--334}.
  

%\bibitem{P}  {Pitman, J.},
 %{Combinatorial stochastic processes},
 %{\sl Lecture Notes in Mathematics},
  %{\bf 1875},
    % {Springer-Verlag, Berlin},
%{2006},  ISBN  978-3-540-30990-1.
   
 \bibitem{EJS}   {Schmutz, Eric},
 {Period lengths for iterated functions},
 {\sl Combin. Probab. Comput.},
 {\bf 20}, {2011}, {no. 2}, {289--298}.
     
\bibitem{Silverman}
 {Silverman, Joseph H.},
 {Variation of periods modulo {$p$} in arithmetic dynamics},
{\sl New York J. Math.}, {\bf 14},
(2008),
 {601--616}.
     
\bibitem{Silvermanbook}
 {Silverman, Joseph H.},
 {The arithmetic of dynamical systems},
 {Graduate Texts in Mathematics},
 {\bf 241},
 {Springer, New York},
(2007), ISBN  {978-0-387-69903-5}.
     
 \bibitem{Stanleyvol1}
{Stanley, Richard P.},
 {Enumerative combinatorics. {V}ol. {I}},
 {The Wadsworth \& Brooks/Cole Mathematics Series},
     {1986}  {0-534-06546-5}.
    

\end{thebibliography}

\end{document}